\documentclass[times,10pt]{amsart}

\newtheorem{thm}{Theorem}

\newtheorem{cor}{Corollary}

\newtheorem{example}{Example}
\newtheorem{lem}{Lemma}
\newtheorem{prop}{Proposition}

\newcommand{\resp}{{\it resp. }}

\renewcommand{\phi}{\varphi}
\renewcommand{\cong}{\simeq}
\newcommand{\ind}{{\rm ind }}
\newcommand{\dif}{{\rm d }}

\newcommand{\ndeg}{{\rm ndeg }}
\newcommand{\ql}{{\rm ql}}
\newcommand{\nd}{{\rm ndeg }}
\newcommand{\M}{\Bbb M}
\newcommand{\F}{\Bbb F}

\begin{document}

\title[]{5-dimensional minimal quadratic and bilinear forms over function fields of conics}

\author{Adam Chapman}
\address{School of Computer Science, Academic College of Tel-Aviv-Yaffo, Rabenu Yeruham St., P.O.B 8401 Yaffo, 6818211, Israel}
\email{adam1chapman@yahoo.com}

\author{Ahmed Laghribi}
\address{Univ. Artois, UR 2462, Laboratoire de Math{\'e}matiques de Lens (LML), F-62300 Lens, France}
\email{ahmed.laghribi@univ-artois.fr}

\begin{abstract}
Over a field of characteristic $2$, we give a complete classification of quadratic and bilinear forms of dimension $5$ that are minimal over the function field of an arbitrary conic. This completes the unique known case due to Faivre concerning the classification of minimal quadratic forms of dimension $5$ and type $(2,1)$ over function fields of nonsingular conics.
\end{abstract}

\date{\today}


\maketitle

\section{Introduction}

Let $F$ be a field of characteristic $2$ and $K/F$ a field extension. An anisotropic $F$-form (quadratic or bilinear) $\phi$ is called {\it $K$-minimal} if $\phi_K$ is isotropic and any form $\psi$ dominated by $\phi$, such that $\dim \psi <\dim \phi$, remains anisotropic over $F$ ($\dim \phi$ denotes the dimension of $\phi$). We refer to Section \ref{background} for the definition of the domination relation which is more refined than the subform relation and is necessary when we take into account singular quadratic forms. Let us mention that the minimality for bilinear forms is equivalent to that of totally singular quadratic forms (Corollary \ref{c1}). Henceforth, we will restrict ourselves on the minimality for quadratic forms.

In the case of a quadratic extension $K/F$, any $K$-minimal form should be of dimension $2$. Obviously, the same conclusion is true when $K$ is the function field of a $2$-dimensional quadratic form. When $K$ is the function field of a conic (singular or not), then any $3$-dimensional anisotropic $F$-form which becomes isotropic over $K$ is necessarily $K$-minimal, and there is no $K$-minimal form of dimension $4$. These two facts combine many references that we summarize below:
\begin{enumerate}
\item[(1)] For quadratic forms $\phi$ of dimension $3$, we use \cite[Th. 1.4]{Laghribi2002}.
\item[(2)] For quadratic forms $\phi$ of dimension $4$, we use:
\begin{enumerate}
\item[(i)] \cite[Th. 1.3]{Laghribi2002} for $\phi$ of type $(2,0)$.
\item[(ii)] \cite[Th. 1.4, subsection 1.4]{Laghribi2002} for $\phi$ of type $(1,2)$.
\item[(iii)] \cite[Th. 1.2]{Laghribi:Rehmann2013} for $\phi$ of type $(0,4)$ and $\ndeg_F(\phi)=8$.
\item[(iv)] \cite[Prop. 1.1 and Th. 1.2]{Laghribi:Rehmann2013}) for $\phi$ of type $(0,4)$ and $\ndeg_F(\phi)=4$.
\end{enumerate}
\end{enumerate}

Here $\ndeg_F(\phi)$ denotes the norm degree of a totally singular quadratic form $\phi$ (see Section \ref{qPf}). Still for $K$ the function field of a conic, $5$-dimensional quadratic forms which are $K$-minimal were classified first in characteristic not $2$ by Hoffmann, Lewis and van Geel \cite{HLVG}. Their result has been extended to characteristic $2$ by Faivre in the case of quadratic forms of type $(2,1)$ and nonsingular conics. His result states the following:

\begin{thm} (Faivre \cite[Prop. 5.2.12]{Faivre2006}) Let $\phi$ be an anisotropic $F$-quadratic form of dimension $5$ and type $(2,1)$, and $\tau=b[1, a]\perp \left<1\right>$ an anisotropic $F$-quadratic form of dimension $3$. Then, $\phi$ is $F(\tau)$-minimal iff the three conditions are satisfied:\\(i) $\phi$ is a Pfister neighbor of a $3$-fold Pfister form $\pi$.\\(ii) $\pi\cong \left< \! \left<c, b, a\right] \! \right]$ for a suitable $c\in F^*$.\\(iii) $\ind (C_0(\phi)\otimes C_0(\tau))=4$.
\label{trappel}
\end{thm}

Moreover, for function fields of nonsingular conics, Faivre proved the following general result: 

\begin{prop} Let $\tau=b[1, a]\perp \left<1\right>$ be an anisotropic $F$-quadratic form of type $(1, 1)$, and $\phi$ an anisotropic $F$-quadratic form. Suppose that $\phi$ is $F(\tau)$-minimal, then we have:\\(1) (\cite[Prop. 5.2.1]{Faivre2006}) $\phi$ is singular but not totally singular.\\(2) (\cite[Prop. 5.2.8]{Faivre2006}) If $\phi$ is of type $(1, l)$, then $l$ is odd.\\(3) \cite[Prop. 5.2.11]{Faivre2006}) If $\dim \phi=5$, then $\phi$ is a Pfister neighbor of $3$-fold Pfister form $\left< \! \left<c, b, a\right] \! \right]$ for some $c\in F^*$.
\label{propg}
\end{prop}

The proof of this proposition is mainly based on the fact that the extension given by the function field of a nonsingular conic is excellent \cite[Cor. 5.7]{Hoffmann:Laghribi2006} and some arguments similar to those developed by Hoffmann, Lewis and Van Geel in characteristic not $2$ \cite{HLVG}. This excellence result is no longer true for singular conics as it was recently proved by Laghribi and Mukhija \cite{Laghribi:Mukhija2024}. 

In our knowledge, no classification of minimal quadratic forms of type $(2,1)$ over function fields of singular conics, or of type $(1,3)$ over function fields of arbitrary conics are known. Our aim in this paper is to complete these open cases. The first result in this sense is the following theorem that concerns minimal quadratic forms of type $(2,1)$ over function fields of singular conics.
\begin{thm}
Let $\phi$ be an anisotropic $F$-quadratic form of dimension $5$ and type $(2,1)$, and $\tau=\left<1, a, b\right>$ an ansiotropic totally singular $F$-quadratic form of dimension $3$. Then, $\phi$ is $F(\tau)$-minimal iff the three  conditions are satisfied:\\(i) $\phi$ is a Pfister neighbor of a $3$-fold Pfister form $\pi$.\\(ii) $\pi\cong \left< \! \left<a, b,c\right] \! \right]$ for some $c\in F^*$.\\(iii) $\ind C_0(\phi)_{F(\sqrt{a}, \sqrt{b})}=2$.
\label{tp1}
\end{thm}

Concerning the classification of minimal $5$-dimensional quadratic forms of type $(1,3)$ over function fields of singular conics, we prove the following result.

\begin{thm}
Let $\phi$ be an anisotropic $F$-quadratic form of type $(1,3)$, and $\tau=\left<1, a, b\right>$ an anisotropic totally singular quadratic form of dimension $3$. Then, $\phi$ is $F(\tau)$-minimal iff the three conditions are satisfied:\\(i) $\phi$ is a Pfister neighbor of a $3$-fold Pfister form $\pi$.\\(ii) $\pi\cong \left< \! \left<a, b,c\right] \! \right]$ for some $c\in F^*$.\\(iii) For any $e\in F^*$, we have either $i_d(e\tau \perp \ql(\phi))\leq 1$ or ($i_d(e\tau \perp \ql(\phi))=2$ and $(D_F(\phi)\cap D_F(e\tau))\setminus D_F(\ql(\phi))=\emptyset$).
\label{tp2}
\end{thm}

The classifications given in Theorems \ref{trappel} et \ref{tp1} are based on the even Clifford algebra $C_0(\phi)$ of $\phi$. However, for quadratic forms of type $(1,3)$, another characterization is used in Theorem \ref{tp2}. This is due to the fact that any quadratic Pfister neighbor of type $(1,3)$ has a split even Clifford algebra as we state in Corollary \ref{coralg}.

\begin{prop}
Let $\phi=a_1[1, b_1]\perp \cdots \perp a_n[1, b_n]\perp \left<1, c_1, \cdots, c_s\right>$ be a singular quadratic form such that $n\geq 1$ and $s\geq 1$, and let $K=F(\sqrt{c_1}, \cdots, \sqrt{c_s})$. Then, $C_0(\phi)$ is isomorphic to the $F$-algebra $[b_1, a_1)\otimes_F\cdots \otimes_F|b_n, a_n)\otimes_FK$. In particular, $C_0(\phi)$ has degree $2^n$ as a $K$-algebra.
\label{pexp}
\end{prop}

As a corollary, we get:

\begin{cor}
Let $\phi$ be an anisotropic $F$-quadratic form of type $(1,3)$ which is a Pfister neighbor. Then, up to a scalar, $\phi$ is isometric to $rs[1, u] \perp \left<1, r, s\right>$ for suitable scalars $r, s, u\in F^*$. Moreover, $C_0(\phi)$ is split as a $K$-algebra, where $K=F(\sqrt{r},\sqrt{s})$.
\label{coralg}
\end{cor}

For the classification of minimal $5$-dimensional quadratic forms of type $(1,3)$ over function fields of nonsingular conics, we prove the following theorem:

\begin{thm}
Let $\phi$ be an anisotropic $F$-quadratic form of type $(1,3)$, and $\tau=a[1, b]\perp \left<1\right>$ an anisotropic quadratic form of dimension $3$ and type $(1,1)$. Then, $\phi$ is $F(\tau)$-minimal iff the three conditions are satisfied:\\(i) $\phi$ is a Pfister neighbor of a $3$-fold Pfister form $\pi$.\\(ii) $\pi\cong \left< \! \left<c, a, b\right] \! \right]$ for some $c\in F^*$.\\(iii) For any $e\in F^*$, if $e[1,b]\subset \phi$ then $e\not\in D_F(a[1, b]).D_F(\ql(\phi))$.
\label{tp4}
\end{thm}

Finally, for the classification of minimal quadratic forms of type $(0,5)$, we use the language of bilinear forms, which will help us to use a cohomological invariant and a classification parallel to those given in Theorems \ref{trappel} et \ref{tp1}. Namely, we will prove the following result:

\begin{thm}
Let $B$ be an anisotropic $F$-bilinear form of dimension $5$, and $Q=\left<1, a, b\right>$ an anisotropic totally singular quadratic form of dimension $3$. Then, the following statements are equivalent:\\(1) $B$ is $F(Q)$-minimal.\\(2) There exists an $F$-bilinear form $C$ of dimension $5$ which is a strong Pfister neighbor of a bilinear Pfister form $\left< \! \left< a, b, c\right> \! \right>_b$ and satisfies the two conditions:\\(i) $\widetilde{B}\cong \widetilde{C}$.\\(ii) For any $u, v\in F^2(a, b)$ such that $\left<u, v, uv\right>_b$ is similar to a subform of $\left< \! \left< a, b, c\right> \! \right>_b$, the invariant $e^2(C \perp \left<\det C\right>_b\perp \left< \! \left<u, v\right> \! \right>_b+I^3F)$ has 
length $2$.
\label{tp3}
\end{thm}
 
Note that for bilinear forms, nothing happens over the function field of a nonsingular conic since an anisotropic bilinear form remains anisotropic over such a field. To clarify the notations used in Theorem \ref{tp3}, let us recall that to any bilinear form $B$ of underlying vector space $V$, we associate a totally singular quadratic form $\widetilde{B}$ defined on $V$ by: $\widetilde{B}(v)=B(v, v)$ for all $v\in V$. The cohomological invariant $e^2$ is that due to Kato \cite{Kato1982} going from the quotient $I^2F/I^3F$ to $\nu_F(2)$, where $\nu_F(2)$ is the additive group generated by the logarithmic symbols $\frac{\dif a_1}{a_1}\wedge \frac{\dif a_2}{a_2}$ for $a_1, a_2\in F^*$. This invariant plays the role of the Clifford invariant which is not defined for bilinear forms in characteristic $2$, and thus the group $\nu_F(2)$ can be seen as the $2$-torsion of the Brauer group. The word ``length'' that we talk about in the previous theorem concerns the smallest number of logarithmic symbols needed to write the cohomological invariant $e^2(\eta)$ for $\eta \in I^2F/I^3F$. Finally, the notion of a strongly Pfister neighbor bilinear form is defined as the classical notion of a quadratic Pfister neighbor. We use the term ``strongly'' since we have another weaker notion of bilinear Pfister neighbor (see below).   

\section{Background on quadratic and bilinear forms}\label{background}

Any quadratic form $\phi$ decomposes as follows:
\begin{equation}
\phi \cong [a_1, b_1]\perp \cdots \perp [a_r, b_r]\perp \left<c_1\right>\perp \cdots \perp \left<c_s\right>,
\label{eqdec}
\end{equation}where $[a, b]$ (\resp $\left<c\right>$) denotes the binary quadratic form $ax^2+xy+by^2$ (\resp $cz^2$). Here, $\cong$ and $\perp$ denote the isometry and the orthogonal sum, respectively. 

As in (\ref{eqdec}), the form $\phi$ is called of type $(r, s)$. We say that $\phi$ is nonsingular (\resp totally singular) if $s=0$ (\resp $r=0$). It is called singular if $s>0$. 

The form $\left<c_1\right>\perp \cdots\perp \left<c_s\right>$ in (\ref{eqdec}) is unique up to isometry, we call it the quasilinear part of $\phi$ and denote it by $\ql(\phi)$. 

The Witt decomposition asserts that $\phi$ uniquely decomposes as follows:
\begin{equation}
\phi\cong \phi_{an} \perp [0,0]\perp \cdots \perp [0,0]\perp \left<0\right>\perp \cdots \perp \left<0\right>,
\label{eqdec2}
\end{equation}where the form $\phi_{an}$ is anisotropic that we call the anisotropic part of $\phi$. The number of copies of the hyperbolic plane $[0, 0]$ in (\ref{eqdec2}) is called the Witt index of $\phi$, we denote it by $i_W(\phi)$. Similarly, the number of $\left<0\right>$ in (\ref{eqdec2}) is called the defect index of $\phi$ and denoted by $i_d(\phi)$. The total index of $\phi$ is $i_W(\phi)+i_d(\phi)$.

Let $C(\phi)$ (\resp $C_0(\phi)$) denote the Clifford algebra (\resp the even Clifford algebra) of the quadratic form $\phi$. When $\phi \cong a_1[1, b_1]\perp \cdots \perp a_r[1, b_r]$ is nonsingular, its Arf invariant $\Delta(\phi)$ is defined as the class of $\sum_{i=1}^r b_i$ in $F/\wp(F)$, where $\wp(F)=\{x^2+x\mid x\in F\}$. In this case, we have $C(\phi)\simeq\otimes_{i=1}^r[b_i, a_i)$, where $[b, a)$ denotes the quaternion algebra generated by two elements $i$ and $j$ subject to the relations: $i^2=a \in F^*$, $j^2+j=b\in F$ and $iji^{-1}=j+1$. 

Two quadratic forms $\phi$ and $\psi$ are called Witt equivalent if $\phi \perp m\times [0,0]\cong \psi \perp n\times [0,0]$ for some integers $m, n\geq 0$. In this case, we write $\phi \sim \psi$.

An $F$-quadratic form $(V, \phi)$ is said to be dominated by  another quadratic form $(W, \psi)$, denoted by $\phi \prec \psi$, if there exists an injective $F$-linear map $\sigma:V\longrightarrow W$ such that $\phi(v)=\psi(\sigma(v))$ for any $v\in V$. We say that $\phi$ is a subform of $\psi$, denoted by $\phi \subset \psi$, if $\psi\cong \phi \perp \phi'$ for a suitable quadratic form $\phi'$. Clearly, if $\phi$ is a subform of $\psi$, then it is dominated by $\psi$, but the converse is not true in general. We refer to 
\cite[Lemma 3.1]{Hoffmann:Laghribi2004} for more details on the domination relation. 

For $a_1, \cdots, a_n\in F^*:=F\setminus \{0\}$, let $\left<a_1, \cdots, a_n\right>_b$ denote the diagonal bilinear form given by:$$((x_1, \cdots, x_n), (y_1, \cdots, y_n))\mapsto \sum_{i=1}^na_ix_iy_i.$$

A metabolic plane is a $2$-dimensional bilinear form isometric to $\begin{pmatrix} a & 1\\1 &0\end{pmatrix}$ for some $a\in F$. A bilinear form is called metabolic if it is isometric to a sum of metabolic planes. 

Let $W_q(F)$ (\resp $W(F)$) be the Witt group of nonsingular quadratic forms (\resp the Witt ring on nondegenarate symmetric bilinear forms). For any integer $n\geq 1$, let $I^nF$ be the $n$-th power of the fundamental ideal $IF$ of $W(F)$ (we take $I^0F=W(F)$). Recall that $W_q(F)$ is endowed with $W(F)$-module structure in a natural way \cite{B}. For any integer $m\geq 2$, let $I^m_qF$ be the submodule $I^{n-1}F\cdot W_q(F)$ of $W_q(F)$. The ideal $I^nF$ is additively generated by the $n$-fold bilinear Pfister forms $\left<1, a_1\right>_b\otimes\cdots \otimes \left<1, a_n\right>_b$, that we denote by $\left< \! \left<a_1, \cdots, a_n\right> \! \right>_b$. The submodule $I^m_qF$ is generated, as a $W(F)$-module, by the quadratic forms $\left< \! \left<a_1, \cdots, a_{m-1}\right> \! \right>_b\cdot [1, b]$, that we denote by $\left< \! \left<a_1, \cdots, a_{m-1}, b\right] \! \right]$ and call an $m$-fold quadratic Pfister form. 

Let $P_mF$ be the set of forms isometric to $m$-fold quadratic Pfister forms, and $GP_mF$ the set of forms similar to forms in $P_mF$. Similarly, let $BP_mF$ be the set of forms isometric to $m$-fold bilinear Pfister forms, and $GBP_mF$ the set of forms similar to forms in $BP_mF$.

Recall that a quadratic (\resp bilinear) Pfister form $Q$ is isotropic iff it is hyperbolic (\resp metabolic). Such a form is also round meaning that $\alpha \in F^*$ is represented by $Q$ iff $Q\cong \alpha Q$.

For a quadratic form $\phi \cong [a_1, b_1]\perp \cdots \perp [a_r, b_r]\perp \left<c_1\right>\perp \cdots \perp \left<c_s\right>$, let take the polynomial $P_{\phi}=\sum_{i=1}^r(a_ix_i^2+x_iy_i+b_iy_i^2)+\sum_{j=1}^sc_jz_j^2$. This polynomial is reducible when $\phi\cong [0,0]\perp \left<0\right>\perp \cdots \perp \left<0\right>$ or $\phi \cong \left<a\right>\perp \left<0\right>\perp \cdots \perp \left<0\right>$ for some $a\in F^*$ \cite[Prop. 3]{MTW}. When $P_{\phi}$ is irreducible, we denote by $F(\phi)$ the quotient field of the ring $F[x_i,y_i, z_j]/(P_{\phi})$, that we call the function field of $\phi$. When $P_{\phi}$ is reducible or $\phi$ is the zero form, then we take $F(\phi)=F$.

A quadratic form $\phi$ is called a Pfister neighbor if there exist $\pi \in P_mF$ and $a\in F^*$ such that $a\phi \prec \pi$ and $2\dim \phi >\dim \pi$. Recall that the form $\pi$ is unique up to isometry, and for any field extension $K/F$ we have that $\phi_K$ is isotropic iff $\pi_K$ is isotropic. In particular, $\phi_{F(\pi)}$ and $\pi_{F(\phi)}$ are isotropic.

\section{Proof of Theorem \ref{tp1}}

\begin{sloppypar}
We start by proving the following proposition that improves \cite[Th. 1.2(3)]{Laghribi2002}.

\begin{prop}
Let $\phi$ be an anisotropic quadratic form of type $(2,1)$, and $\tau=\left<1, a, b\right>$ an anisotropic totally singular quadratic form of dimension $3$. Suppose that $\phi$ is not a Pfister neighbor. Then, $\phi_{F(\tau)}$ is isotropic iff $\tau$ is dominated by $\phi$ up to a scalar.
\label{p1}
\end{prop}

\begin{proof}
It is clear that if $\tau$ is dominated by $\phi$ up to a scalar, then $\phi_{F(\tau)}$ is isotropic. Conversely, suppose that $\phi_{F(\tau)}$ is isotropic. It follows from \cite[Th. 1.2(3)]{Laghribi2002} that 
\begin{equation}
\phi \sim x\pi \perp \phi'
\label{e1}
\end{equation}where $x\in F^*$, $\pi$ a $3$-fold Pfister form isotropic over $F(\tau)$, and $\phi'$ a form of type $(2,1)$ that dominates $\tau$ up to a scalar. Without loss of generality, we may suppose $\phi \cong R\perp \left<1\right>$, where $R$ is nonsingular of dimension $4$. Hence, $\left<1\right>$ is also the quasilinear part of $\phi'$. 

Let $y\in F^*$ be such that $y\tau$ is dominated by $\phi'$. Hence, by the domination relation, there exist $u, v, w\in F^*$ such that $y\tau\cong \left< u, v, w\right>$ and $\phi'\cong u[1, p] \perp v[1, q]\perp \left<w\right>$ for suitable $p, q\in F^*$. By the uniqueness of the quasilinear part, we have $\left<1\right>\cong \left<w\right>$. Without loss of generality, we may suppose $w=1$. Moreover, the isotropy of $\pi_{F(\tau)}$ implies that $\pi\cong \left< \! \left<a, b, c\right] \! \right]$ for some $c\in F^*$. 

We add in both sides of (\ref{e1}) the form $\theta:=y\left< \! \left<a, b\right> \! \right>$, we get:

\begin{equation}
R \perp \theta \perp \left<0\right> \sim x\pi \perp \theta \perp \left<0\right>.
\label{e2}
\end{equation}
Cancelling the form $\left<0\right>$ and comparing the dimension in both sides of (\ref{e2}), we see that $x\pi \perp \theta$ is isotropic. By the roundness of (quasi-)Pfister forms, there exists $z\in F^*$ such that $x\pi \perp \theta \cong z(\pi \perp \left< \! \left<a, b\right> \! \right>)\sim z\left< \! \left<a, b\right> \! \right>$.  Hence, we get
\begin{equation}
R \perp \theta  \sim z\left< \! \left<a, b\right> \! \right>.
\label{e3}
\end{equation}This implies that the Witt index of $R \perp \theta$ is $2$. Hence, there exists a form $\left<r, s\right>$ which is dominated by $R$ and $\theta$  \cite[Prop. 3.11]{Hoffmann:Laghribi2004}. Then, $\left<1, r, s\right>$ is dominated by $\phi$. In particular, $\left<1, r, s\right>$ is anisotropic. Since $\theta$ represents $r$, $s$ and $1$ (because we take $w=1$), it follows that $\left<1, r, s\right>$ is dominated by $\theta$. Consequently, $\left<1, r, s\right>$ becomes isotropic over $F(\tau)$ because $\tau$ and $\left<1, r, s\right>$ are quasi-Pfister neighbors of the same quasi-Pfister form $\left< \! \left<a, b\right> \! \right>$. It follows from \cite[Th. 1.2(1)]{Laghribi:Rehmann2013} that $\tau$ is similar to $\left<1, r, s\right>$, and thus $\tau$ is dominated by $\phi$ up to a scalar.\qed
\end{proof}

As an immediate corollary, we get:

\begin{cor}
Let $\phi$ be an anisotropic quadratic form of type $(2,1)$, and $\tau=\left<1, a, b\right>$ an anisotropic totally singular quadratic form of dimension $3$. If $\phi$ is $F(\tau)$-minimal, then $\phi$ is a Pfister neighbor.
\label{cor1}
\end{cor}

\noindent{\bf Proof of Theorem \ref{tp1}.} Let $\phi$ be an anisotropic $F$-quadratic form of dimension $5$ and type $(2,1)$, and $\tau=\left<1, a, b\right>$ an ansiotropic totally singular $F$-quadratic form of dimension $3$.

(2) $\Longrightarrow$ (1) Suppose that the conditions (i), (ii) and (iii) given in the theorem are satisfied. Since $\pi_{F(\tau)}$ is isotropic and $\phi$ is a Pfister neighbor of $\pi$, it follows that $\phi_{F(\tau)}$ is isotropic. Suppose that $\phi$ is not $F(\tau)$-minimal. Then, there exists $\psi$ a form dominated by $\phi$ of dimension $3$ or $4$ such that $\psi_{F(\tau)}$ is isotropic. Using \cite[Th. 1.3]{Laghribi2002}, we can see that the form $\psi$ is neither of type $(2, 0)$ nor of type $(1, 1)$. Hence, $\psi$ is of type $(0, 3)$ or of type $(1, 2)$. But in both cases, there exists $x\in F^*$ such that the form $x\tau$ is dominated by $\psi$ (use \cite[Th. 1.4]{Laghribi2002} for type $(1, 2)$, and \cite[Th. 1.2]{Laghribi:Rehmann2013} for type $(0, 3)$). There exist $u, v, w$ such that $x\tau\cong \left<u, v, w\right>$ and $\phi \cong u[1, p]\perp v[1, q]\perp \left<w\right>$. We have that $C_0(\phi)$ is Brauer equivalent to $[p, uw)\otimes_F[q, vw)$. Because $x\left<1, a, b\right>\cong \left<u, v, w\right>$, the scalars $uw$ and $vw$ are squares $F(\sqrt{a}, \sqrt{b})$. Consequently, $C_0(\phi)_{F(\sqrt{a}, \sqrt{b})}$ is split, a contradiction.

(1) $\Longrightarrow$ (2) Suppose that $\phi$ is $F(\tau)$-minimal. By Corollary \ref{cor1}, we deduce that $\phi$ is a Pfister neighbor of a $3$-fold Pfister form $\pi$. By \cite[Prop. 3.2]{Laghribi2002}, we have $\ind C_0(\phi)\leq 2$. Since $\phi$ is anisotropic, we get $\ind C_0(\phi)= 2$. The isotropy of $\phi_{F(\tau)}$ implies that $\pi_{F(\tau)}$ is also isotropic, and thus hyperbolic. Hence, $\pi \cong \left< \! \left<a, b, c\right] \! \right]$ for some $c\in F^*$. Without loss of generality, we may suppose that $\phi=R\perp \left<1\right>$ for $R$ nonsingular of dimension $4$. Let $x\in F^*$ be such that $x\phi$ is dominated by $\pi$. Hence, we get 
\begin{equation}
x\pi \cong R \perp [1, r] \perp s[1, t]
\label{e4}
\end{equation}for suitable $r, s\in F^*$ and $t=r+\Delta(R)\in F/\wp(F)$. Taking the Clifford algebra of $x\pi$, we get $C(R)=[t, s)$. Moreover, we have $C(R)\sim C_0(\phi)$ \cite[Lemma 2]{MTW}. 

Suppose that $C_0(\phi)_{F(\sqrt{a}, \sqrt{b})}$ is split. Then, $\left< \! \left<s, t\right] \! \right]$ is hyperbolic over $F(\sqrt{a}, \sqrt{b})$. By a result of Baeza \cite[Corollary 4.16, page 128]{B}, we get $\left< \! \left<s, t\right] \! \right]\sim \left<1, a\right>\otimes \phi_1 \perp \left<1, b\right>\otimes \phi_2$ for suitable $\phi_1, \phi_2\in W_q(F)$. Hence, we get$$\left< \! \left<s, t\right] \! \right]\perp \left< \! \left<a, k_1\right] \! \right] \perp \left< \! \left<b, k_2\right] \! \right] \in I^3_qF,$$where $k_i=\Delta(\phi_i)$ for $i=1, 2$. Combining with equation (\ref{e4}) yields
\begin{equation}
R\perp [1, r+t+k_1+k_2] \perp a[1, k_1]\perp b[1, k_2] \in I^3_qF.
\label{e5}
\end{equation}It follows from Pfister's result that there exists $\rho\in GP_3F$ such that$$R\perp [1, r+t+k_1+k_2] \sim a[1, k_1]\perp b[1, k_2] \perp \rho.$$Consequently, by adding $\left<1\right>$ on both sides, we get$$\phi\sim a[1, k_1]\perp b[1, k_2] \perp \left<1\right>\perp \rho.$$Clearly, from this last equivalence, we deduce that $\rho_{F(\tau)}$ is isotropic. Now we have the same conditions as in equation (\ref{e1}) in the proof of Proposition \ref{p1}. We reproduce the same arguments to deduce that $\tau$ is dominated by $\phi$ up to a scalar, a contradiction to the minimality of $\phi$.\qed
\vskip1.5mm

Let us give an example for which Theorem \ref{tp1} applies. This example is similar to the one given by Chapman and Qu\'eguiner-Mathieu for the minimality over the function field of a nonsingular conic \cite{CQM}.
\vskip1.5mm

\begin{example} Let $a, b, c$ be indeterminates over a field $F_0$ of characteristic $2$. Let us consider the forms $\phi=c[1, a+b]\perp b[1, a]\perp \left<1\right>$, $\tau=\left<1, b, ac\right>$ and $\pi=\left< \! \left<b, c, a+b\right] \! \right]$ over the rational function field $F=F_0(a, b, c)$. We have the following:\\(1) $\pi\cong \left<c, bc\right>_b.[1, a+b] \perp \left<1, b\right>_b.[1, a]$ because $\left<1, b\right>_b.[1, a+b]\cong \left<1, b\right>_b.[1, a]$. This proves that $\phi \prec \pi$, and thus $\phi$ is a Pfister neighbor of $\pi$.\\(2) $\phi_{F(\tau)}$ is isotropic because $c\tau \prec \pi$ as we can see from (1).\\(3) Let $L=F(\sqrt{b}, \sqrt{abc})$. We have in the Brauer group of $F$ the following:$$C_0(\phi)_L=[a+b, bc)_L=[a+b,a)_L=[b, a)_L.$$The algebra $[b, a)$ is division over $F_1:=F_0(a,b)(\sqrt{b})$, and it remains division over $L=F_1(\sqrt{ac})$.\\Hence, Theorem \ref{tp1} implies that $\phi$ is $F(\tau)$-minimal.
\label{exa1}
\end{example}

\section{Proof of Theorem \ref{tp2}}

As we did for the proof of Theorem \ref{tp1}, we start proving a result on the isotropy that we need for the proof of Theorem \ref{tp2}. This result refines \cite[Th. 1.1(3)]{Dolphin:Laghribi2017}.

\begin{prop} Let $\phi$ be an anisotropic $F$-quadratic form of dimension $5$ and type $(1, 3)$. Let $\tau=\left<1, a, b\right>$ be an anisotropic totally singular $F$-quadratic form of dimension $3$. Suppose that $\phi$ is not a Pfister neighbor. Then, $\phi$ is isotropic over $F(\tau)$ iff $\tau$ is dominated by $\phi$ up to a scalar.
\label{priso}
\end{prop}

\begin{proof} The proof of this proposition looks like that of Proposition \ref{p1}. However, few changes will be needed to adapt it to the type $(1,3)$, this is why we include a proof. 

Suppose that $\phi_{F(\tau)}$ is isotropic. If $\tau$ is similar to $\ql(\phi)$ then we are done. So suppose that $\tau$ is not similar to $\ql(\phi)$. Since $\phi_{F(\tau)}$ is isotropic, we get from \cite[Th. 1.1(3)]{Dolphin:Laghribi2017}   
\begin{equation}
\phi \sim x\pi \perp \phi'
\label{2e1}
\end{equation}where $x\in F^*$, $\pi$ a $3$-fold Pfister form isotropic over $F(\tau)$, and $\phi'$ a form of type $(1,3)$ that dominates $\tau$ up to a scalar. Without loss of generality, we may write $\phi \cong R\perp \left<1, r, s\right>$, where $R$ is nonsingular of dimension $2$. Hence, $\left<1, r, s\right>$ is also the quasilinear part of $\phi'$.

Let $y\in F^*$ be such that $y\tau$ is dominated by $\phi'$. Hence, by the domination relation and using the fact that $\tau$ and $\ql(\phi)=\ql(\phi')$ are not similar, there exist $u, v, w, z\in F^*$ such that $y\tau\cong \left< u, v, w\right>$ and $\phi'\cong u[1, p] \perp \left<v, w, z\right>$. Moreover, the isotropy of $\pi_{F(\tau)}$ implies that $\pi\cong \left< \! \left<a, b, c\right] \! \right]$ for some $c\in F^*$. 

We add in both sides of (\ref{2e1}) the form $\theta:=y\left< \! \left<a, b\right> \! \right>$, we get:

\begin{equation}
R \perp \left<z\right>\perp \theta \perp \left<0, 0\right> \sim x\pi \perp \left<z\right>\perp  \theta \perp \left<0, 0\right>.
\label{2e2}
\end{equation}
Cancelling the form $\left<0, 0\right>$ and comparing the dimension of both sides in (\ref{2e2}), we see that $x\pi \perp \theta$ is isotropic. By the roundness of (quasi-)Pfister forms, there exists $s\in F^*$ such that $x\pi \perp \theta \cong s(\pi \perp \left< \! \left<a, b\right> \! \right>)\sim s\left< \! \left<a, b\right> \! \right>$.  Hence, we get
\begin{equation}
R \perp \left<z\right>\perp \theta  \sim \left<z\right>\perp s\left< \! \left<a, b\right> \! \right>.
\label{2e3}
\end{equation}This implies that $R \perp \left<z\right>\perp \theta$ is isotropic. Hence, there exists $t\in D_F(R\perp \left<z\right>)\cap D_F(\theta)$. Then, $\left<v, w, t\right>$ is dominated by $\phi$. In particular, $\left<v, w, t\right>$ is anisotropic. Since $\theta$ represents $v$, $w$ and $t$, it follows that $\left<v, w, t\right>$ is dominated by $\theta$. Consequently, $\left<v, w, t\right>$ becomes isotropic over $F(\tau)$ because $\tau$ and $\left<v, w, t\right>$ are quasi-Pfister neighbors of the same quasi-Pfister form $\left< \! \left<a, b\right> \! \right>$. It follows from \cite[Th. 1.2(1)]{Laghribi:Rehmann2013} that $\tau$ is similar to $\left<v, w, t\right>$, and thus $\tau$ is dominated by $\phi$ up to a scalar.\qed
\end{proof}
\medskip
The following corollary is immediate:

\begin{cor}
Let $\phi$ be an anisotropic $F$-quadratic form of dimension $5$ and type $(1, 3)$, and $\tau=\left<1, a, b\right>$ an anisotropic totally singular $F$-quadratic form of dimension $3$. If $\phi$ is $F(\tau)$-minimal, then it is a Pfister neighbor.
\label{c2}
\end{cor}

\noindent{\bf Proof of Theorem \ref{tp2}.} Let $\phi$ be an anisotropic $F$-quadratic form of type $(1,3)$, and $\tau=\left<1, a, b\right>$ an anisotropic totally singular quadratic form of dimension $3$. 

-- Suppose that $\phi$ is $F(\tau)$-minimal. Then, by Corollary \ref{c2}, $\phi$ is a Pfister neighbor of a $3$-fold Pfister form $\pi$. Since $\phi_{F(\tau)}$ isotropic, then $\pi_{F(\tau)}$ is also isotropic. Hence, $\pi \cong \left< \! \left<a, b, c\right] \! \right]$ for some $c\in F^*$. Suppose that there exists $e\in F^*$ such that: $i_d(e\tau \perp \ql(\phi))\geq 2$ and ($i_d(e\tau \perp \ql(\phi))\neq 2$ or $(D_F(\phi)\cap D_F(e\tau))\setminus  D_F(\ql(\phi))\neq\emptyset$). This is equivalent to saying: $i_d(e\tau \perp \ql(\phi))=3$ or ($i_d(e\tau \perp \ql(\phi))\geq 2$ and $(D_F(\phi)\cap D_F(e\tau))\setminus  D_F(\ql(\phi))\neq\emptyset$). The condition $i_d(e\tau \perp \ql(\phi))=3$ means that $\ql(\phi)$ is similar to $\tau$, while the second condition means that $e\tau$ is dominated by $\phi$. Hence, $\phi$ is not $F(\tau)$-minimal, a contradiction. 

-- Conversely, suppose we have the three conditions (i), (ii) and (iii) as described in the theorem. Since $\phi$ is a Pfister neighbor of a $3$-fold Pfister form $\left< \! \left<a, b, c\right] \! \right]$ for some $c\in F^*$, it follows that $\phi_{F(\tau)}$ is isotropic. Suppose that $\phi$ is not $F(\tau)$-minimal. Then, there exists $\psi$ a form dominated by $\phi$ of dimension $3$ or $4$ such that $\psi_{F(\tau)}$ is isotropic. Then, $e\tau$ is dominated by $\psi$ for a suitable $e\in F^*$ (we use \cite[Th. 1.2]{Laghribi:Rehmann2013} when $\psi$ is totally singular, and \cite[Th. 1.4, subsection 1.4]{Laghribi2002} when $\psi$ is of type $(1,2)$). This gives two possibilities: 

(a) $e\tau$ is isometric to $\ql(\phi)$, which contradicts the condition (iii).

(b) or there exists $x, y, z, t, u\in F^*$ such that $e\tau\cong \left<x, y, z\right>$ and $\phi \cong x[1, u]\perp \left<y, z, t\right>$. This condition also contradicts (iii) because $i_d(e\tau \perp \ql(\phi))=2$ but $x\in (D_F(\phi)\cap D_F(e\tau))\setminus D_F(\ql(\phi)$. Hence, $\phi$ is $F(\tau)$-minimal.\qed
\medskip

\noindent{\bf Proof of Proposition \ref{pexp}.} Suppose $\phi=a_1[1, b_1]\perp \cdots \perp a_n[1, b_n]\perp \left<1, c_1, \cdots, c_s\right>$ for $a_i, b_i, c_j\in F$ such that $a_i\neq 0$. The Clifford algebra is generated by $x_1,y_1,\dots,x_n,y_n,z_1,\dots,z_s$ such that $z_i$ commutes with all the generators, $x_i$ commutes with $y_j$ when $i\neq j$, and $x_iy_i+y_ix_i=1$ and $x_i^2=a_i$, $y_i^2=\frac{b_i}{a_i}$ and $z_i^2=c_i$.
The even Clifford algebra is generated by $u_1,v_1,\dots,u_n,v_n,w_2,\dots,w_s$ where $u_i=x_iz_1$, $v_i=y_i z_1$ and $w_i=z_i z_1$. The relations are the following: $w_i$ commutes with all the other generators, $u_i$ commutes with $v_j$ for $i\neq j$, $u_iv_i+v_iu_i=z_1^2=1$ and $u_i^2=a_i$, $v_i^2=\frac{b_i}{a_i}$ and $w_i^2=c_i$.
Therefore, the even Clifford algebra is 
$$F\langle u_1,v_1 \rangle \otimes F \langle u_2,v_2 \rangle \otimes \dots \otimes F \langle u_n,v_n \rangle \otimes F \langle w_2,\dots,w_n \rangle,$$
which is indeed $[b_1,a_1)_F \otimes \dots \otimes [b_n,a_n)_F \otimes K$.\qed

\medskip

\noindent{\bf Proof of Corollary \ref{coralg}.} Let $\phi$ be an anisotropic quadratic form of type $(1, 3)$. Suppose that $\phi$ is a Pfister neighbor of a $3$-fold Pfister form $\pi$. Modulo a scalar, we may write $\phi=R\perp \left<1, r, s\right>$ for suitable nonsingular quadratic form $R$ of dimension $2$ and $r, s\in F^*$. On the one hand, since $\pi$ is isotropic over $F(\left<1, r, s\right>)$, it follows that $\pi$ is also isotropic over $F(\left< \! \left<r, s\right> \! \right>)$, and thus $\pi \cong \left< \! \left<r, s,u\right] \! \right]$ for some $u\in F^*$. On the other hand, the hyperbolicity of $\pi_{F(\phi)}$ implies that $\pi \cong R\perp [1, x]\perp r[1, y] \perp s[1, z]$ for some $x, y, z\in F^*$. Hence, we get$$\left< \! \left<r, s,u\right] \! \right]\cong R\perp [1, x]\perp r[1, y] \perp s[1, z].$$Adding on both sides in the last isometry the form $\left<1, r, s\right>$, and canceling the hyperbolic planes, yields that $\phi \cong rs[1, u]\perp \left<1,r, s\right>$.

The fact that $C_0(\phi)$ is split as an $F(\sqrt{r}, \sqrt{s})$-algebra is a direct consequence of Proposition \ref{pexp}.\qed 

We finish this section by an example applying the criteria given in Theorem \ref{tp2}.

\begin{example}
Let $r, s, u$ be indeterminates over a field $F_0$ of characteristic $2$. Let us consider the forms $\phi=rs[1, u]\perp \left<1, r, s\right>$ and $\tau=\left<1, ru, s(r^2+r+u)\right>$ over the rational function field $F=F_0(r, s, u)$. We have the following statements:\\(1) It is clear that $\phi$ is a Pfister neighbor of $\pi=\left< \! \left<r, s, u\right] \! \right]$. Moreover, $\tau$ is dominated by $\pi$ because the scalars $1$, $ru$ and $s(r^2+r+u)$ are represented by the forms $[1, u]$, $r[1, u]$ and $s[1, u]$, respectively. Hence, $\phi_{F(\tau)}$ is isotropic.\\(2) For any $e\in F^*$, we have $i_d(e\tau \perp \ql(\phi))\leq 1$. In fact, suppose that $i_d(e\tau \perp \ql(\phi)\geq 2$ for some $e\in F^*$. Then, using \cite[Prop. 3.2]{Hoffmann:Laghribi2004}, there exists a totally singular quadratic form of dimension $2$ which is dominated by $e\tau$ and $\ql(\phi)$. Consequently, there exists an inseparable quadratic extension $K=F(\sqrt{d})$ such that $\left< \! \left<r, s\right> \! \right>_K$ and $\left< \! \left<ru, s(r^2+r+u)\right> \! \right>_K$ are isotropic, and thus quasi-hyperbolic. This implies that 
$\left< \! \left<r, s\right> \! \right>\cong \left< \! \left<d, k\right> \! \right>$ and $\left< \! \left<ru, s(r^2+r+u)\right> \! \right>\cong \left< \! \left<d, l\right> \! \right>$ for suitable $k, l\in F^*$. Hence, $\left< \! \left<r, s\right> \! \right>\perp \left< \! \left<ru, s(r^2+r+u)\right> \! \right>$ has defect index $\geq 2$. In particular, $\theta:=\left< \! \left<r, s\right> \! \right>\perp \left<ru, s(r^2+r+u), rus(r^2+r+u)\right>$ is isotropic. But, using the classical isometry $\left<a, b\right>\cong \left<a, a+b\right>$ for any $a, b\in F$, we get:
\begin{eqnarray*}
\theta & = & \left<1, r, s, rs\right> \perp \left<ru, r^2s+rs+ su, rsu(r^2)+su(r^2)+rs(u^2)\right>\\&\cong & 
\left<1,r, s, rs\right> \perp \left<ru, su, rsu(r^2)+su(r^2)+rs(u^2)\right>\\& \cong & \left<1,r, s, rs\right> \perp \left<ru, su, rsu(r^2)\right>\\& \cong & \left<1,r, s, rs\right> \perp \left<ru, su, rsu\right>\\& =& \left<1\right>\perp \left<1, u\right>.\left<r, s, rs\right>,
\end{eqnarray*}which shows that $\theta$ is anisotropic, a contradiction. Hence, $i_d(e\tau \perp \ql(\phi))\leq 1$, and thus Theorem \ref{tp2} implies that $\phi$ is $F(\tau)$-minimal.\qed
\label{exa2}
\end{example}

\section{Proof of Theorem \ref{tp4}} 
Following Proposition \ref{propg}, any anisotropic quadratic form of type $(1,3)$ which is minimal over the function field of a nonsingular conic is necessary a Pfister neighbor. The proof given by Faivre uses some arguments developed by Hoffmann, Lewis and Van Geel in characteristic not $2$ \cite{HLVG}. In this section, we give an alternative proof of this fact using a different method based on the following result on isotropy that refines \cite[Th. 1.1(1)]{Dolphin:Laghribi2017}.

\begin{prop}
Let $\phi$ be an anisotropic quadratic form of type $(1,3)$, and $\tau$ an anisotropic quadratic form of type $(1,1)$. Suppose that $\phi$ is not a Pfister neighbor. Then, $\phi_{F(\tau)}$ is isotropic iff $\tau$ is dominated by $\phi$ up to a scalar.
\label{pinsc}
\end{prop}

\begin{proof} Obviously, if $\tau$ is dominated by $\phi$ up to a scalar, then $\phi_{F(\tau)}$ is isotropic. Conversely, suppose that $\phi$ is isotropic over $F(\tau)$. It follows from \cite[Th. 1.1]{Dolphin:Laghribi2017} that there exist $\alpha, \beta, u, v\in F^*$ and $R_1, R_2$ nonsingular quadratic forms of dimension $2$ such that $\alpha\phi\cong R_1\perp \left<1, u, v\right>$, $\beta \tau=R_2\perp \left<1\right>$ and
\begin{equation}
R_1\perp R_2 \perp \rho\sim x\pi,
\label{nsceq1}
\end{equation}where $x\in F^*$, $\rho$ is a nonsingular complement of $\left<1, u, v\right>$ and $\pi \in P_3F$ dominates $\tau$ up to a scalar.

Let us write $R_2=b[1, a]$ for suitable $a, b\in F^*$. Since $\pi_{F(\tau)}$ is isotropic, then $\pi \cong \left< \! \left<b,a\right] \! \right] \perp y\left< \! \left<b,a\right] \! \right]$ for some $y\in F^*$. Adding to both sides of equation (\ref{nsceq1}) the form $\pi$, we get
\begin{equation}
R_1 \perp \rho \perp [1, a]\perp y\left< \! \left<b,a\right] \! \right]\sim x\pi\perp \pi \in I^4_qF.
\label{nsceq2}
\end{equation}
Since the left hand side of equation (\ref{nsceq2}) is of dimension $14$, it follows from the Hauptsatz \cite[Prop. 6.4]{Laghribi2011} that
\begin{equation}
R_1 \perp \rho \perp [1, a]\perp y\left< \! \left<b,a\right] \! \right]\sim 0. 
\label{nsceq3}
\end{equation}
Adding to both sides of equation (\ref{nsceq3}) the form $\left<1, u, v\right>$ and using the equivalences $\rho \perp \left<1, u, v\right>\sim \left<1, u, v\right>$ and $[1, a]\perp \left<1\right>\sim \left<1\right>$, we get
\begin{equation}
\alpha\phi \sim  y\left< \! \left<b,a\right] \! \right]\perp \left<1, u, v\right>. 
\label{nsceq4}
\end{equation}
This implies that $y\left< \! \left<b,a\right] \! \right]\perp \left<1, u, v\right>$ is isotropic. Hence, there exists $w\in D_F(\left<1, u, v\right>)\cap D_F(y\left< \! \left<b,a\right] \! \right])$. Since $\left< \! \left<b, a\right] \! \right]$ is round and $\left<1, u, v\right>\cong \left<w,\cdots\right>$, equation (\ref{nsceq4}) becomes
\begin{equation}
\alpha\phi \sim  wb[1, a]\perp \left<1, u, v\right>. 
\label{nsceq5}
\end{equation}
Consequently, $\alpha\phi \cong  wb[1, a]\perp \left<1, u, v\right>$, and thus $w\tau$ is dominated by $\phi$.\qed
\end{proof}
\medskip
We get the following corollary.

\begin{cor}
Let $\phi$ be an anisotropic $F$-quadratic form of type $(1, 3)$, and $\tau$ an anisotropic $F$-quadratic form of type $(1, 1)$. If $\phi$ is $F(\tau)$-minimal, then it is a Pfister neighbor.
\label{c3}
\end{cor}

\noindent{\bf Proof of Theorem \ref{tp4}.} Let $\phi$ be an anisotropic $F$-quadratic form of type $(1,3)$, and $\tau=a[1, b]\perp \left<1\right>$ an anisotropic quadratic form of type $(1,1)$.

-- Suppose that $\phi$ is $F(\tau)$-minimal. It follows from Corollary \ref{c3} that $\phi$ is a Pfister neighbor of a $3$-fold Pfister form $\pi$. Since $\phi_{F(\tau)}$ is isotropic, then $\pi_{F(\tau)}$ is isotropic as well. Hence, $\pi\cong \left< \! \left<c, a, b\right] \! \right]$ for some $c\in F^*$. Let $e\in F^*$ be such that $e[1, b]\subset \phi$. Hence, $\phi \cong e[1, b]\perp \ql(\phi)$. Suppose that $e\in D_F(a[1, b]).D_F(\ql(\phi))$. Let $x\in D_F([1, b])$ and $y\in D_F(\ql(\phi))$ be such that $e=axy$. Since $[1, b]\cong x[1, b]$, it follows that $e[1, b]\cong axy[1, b]\cong ay[1, b]$, and thus $y\tau$ is dominated by $\phi$, a contradiction.

-- Conversely, suppose that $\tau$ satisfies the conditions (i), (ii) and (iii) are described in the theorem. The conditions (i) and (ii) imply that $\phi_{F(\tau)}$ is isotropic. Suppose that $\phi$ is not $F(\tau)$-minimal. Hence, there exists a form $\psi$ of dimension $3$ or $4$ dominated by $\phi$ and isotropic over $F(\tau)$. The form $\psi$ is of type $(1,1)$ or $(1, 2)$. We use \cite[Th. 1.4(2), subsection 1.4]{Laghribi2002} when $\dim \psi=3$, and \cite[Th.1.4(2), subsection 1.3]{Laghribi2002} when $\dim \psi=4$ to conclude that $\tau$ is dominated by $\psi$. Hence, there exist $e, f, g\in F^*$ such that $\phi \cong ea[1, b]\perp \left<e, f, g\right>$. This contradicts (iii) because $ea\in D_F(a[1, b]).D_F(\ql(\phi))$.\qed

As we did for the other classifications, we give an example of a minimal form that uses Theorem \ref{tp4}:

\begin{example}
Let $F=\F_2(r, s, u)$ be the rational function field in the indeterminates $r, s, u$ over the field $\F_2$ with two elements. Let $\phi=rs[1, u]\perp \left<1, r, s\right>$, $\pi =\left< \! \left<r, s, u\right] \! \right]$ and $\tau=su[1, r+u]\perp \left<1\right>$. It is clear that $\phi$ is a Pfister neighbor of $\pi$. Moreover, $\pi\cong [1, r+u]\perp [r, ur^{-1}+1]\perp s[1,u]\perp rs[1, u]$, and thus $su\tau \subset \pi$. Hence, $\phi$ is isotropic over $F(\tau)$. Suppose that $\phi$ is not $F(\tau)$-minimal. There exists by Theorem \ref{tp4} a scalar $e\in F^*$ such that $e[1, r+u]\subset \phi$ and $e\in D_F(su[1, r+u]). D_F(\left<1, r, s\right>)$. Hence, using the roundness of $[1, r+u]$, we get $\phi \cong sut[1, r+u]\perp \left<1, r, s\right>$ for a suitable $t\in D_F(\left<1, r, s\right>)$. Let $M=\F_2(r, s)$ and consider the $M$-place $\lambda$ from $F$ to $M$ with respect to the $u$-adic valuation of $F$. We have:\\(1) $t$ is a unit for the $u$-adic valuation because $t\in D_F(\left<1, r, s\right>)$.\\(2) the form $\phi$ has nearly good reduction with respect to $\lambda$ (in the sense of Knebusch \cite[Page 289]{Knebusch1973}) since it is isometric to $\phi=rs[1, u]\perp \left<1, r, s\right>$ and $\left<1, r, s\right>_M$ is anisotropic. On the one hand, the total index of $\lambda_*(\phi)$ is equal to $1$ because $\lambda_*(rs[1, u])=[0,0]$ and $\left<1, r, s\right>_M$ is anisotropic. On the other hand, since $\phi$ contains $sut[1, r+u]$ as a subform, and $\lambda(\alpha)=0$ or $\infty$ for every $\alpha$ represented by $sut[1, r+u]$, we conclude by \cite[Prop. 3.4]{Knebusch1973} that the total index of $\lambda_*(\phi)$ is at least $2$, a contradiction.
\label{exa3}
\end{example}

\section{(Quasi-)Pfister neighbor forms}\label{qPf}
 
Our aim in this section is to relate the notions of quasi-Pfister neighbor and bilinear (strong)Pfister neighbor. This is useful to classify $F(\tau)$-minimal bilinear forms of dimension $5$ for $\tau$ a totally singular quadratic form of dimension $3$.

A {\it quasi-Pfister form} is a totally singular quadratic form $\widetilde{B}$ for some bilinear Pfister form $B$. A totally singular quadratic form $Q$ is called {\it quasi-Pfister neighbor} if there exists a quasi-Pfister form $\pi$ such that $2\dim Q>\dim \pi$ and $aQ\subset \pi$ for some $a\in F^*$. In this case, the form $\pi$ is unique, and for any field extension $K/F$ the form $Q_K$ is isotropic iff $\pi_K$ is isotropic. In particular, $Q_{F(\pi)}$ and $\pi_{F(Q)}$ are isotropic. 

The {\it norm field} of a nonzero totally singular quadratic form $Q$ is the field $N_F(Q):=F^2(\alpha\beta\mid \alpha, \beta \in D_F(Q))$, where $D_F(Q)$ is the set of scalars in $F^*$ represented by $Q$. The degree $[N_F(Q):F^2]$ is called the {\it norm degree} of $Q$ and it is denoted by $\nd_F(Q)$. Clearly, we have $\nd_F(Q)=2^d$ for some integer $d\geq 1$ and $\nd_F(Q)\geq \dim Q$. We refer to \cite[Section 8]{Hoffmann:Laghribi2004} for more details on the norm degree and some of its important applications.

Here is a characterization of quasi-Pfister neighbor using the norm degree.

\begin{prop}(\cite[Proposition 8.9(ii)]{Hoffmann:Laghribi2004}) 
Let $Q$ be an anisotropic totally singular quadratic form. Then, $Q$ is quasi-Pfister neighbor iff $2\dim Q>\nd_F(Q)$.
\label{pp1}
\end{prop}

The norm degree also appears in the description of the Witt kernels for bilinear forms.

\begin{prop}(\cite[Theorem 1.2]{Laghribi2005}) 
Let $B$ be an anisotropic $F$-bilinear form and $Q$ an anisotropic totally singular form of norm degree $2^d$. If $B$ becomes metabolic over $F(Q)$, then $\dim B$ is divisible by $2^d$.
\label{pnd2}
\end{prop}

A bilinear form $B$ is called a {\it Pfister neighbor} if $\widetilde{B}$ is a quasi-Pfister neighbor. This definition does not imply that $B$ is contained in a bilinear Pfister form up to a scalar. For example, over the rational functions field $F(t_1, t_2)$, the bilinear form $B=\left<1, t_1, t_2, 1+t_1t_2\right>_b$ is a Pfister neighbor because $\widetilde{B}\cong \left< \! \left<t_1,t_2\right> \! \right>$, but $B$ is not similar to a subform of a bilinear Pfister form since its determinant is not trivial. We refer to \cite{Laghribi2005} where bilinear Pfister neighbors are studied and some of their splitting properties are given.
\vskip1.5mm
A bilinear form $B$ is called a {\it strong Pfister neighbor}, shortened by SPN, if there exists a bilinear Pfister form $\rho$ such that $2\dim B>\dim \rho$ and $\alpha B\subset \rho$ for some $\alpha \in F^*$. In this case, the form $\rho$ is unique. In fact, if $B$ is a SPN of another bilinear Pfister form $\delta$, then there exists $\beta \in F^*$ such that $\beta B\subset \delta$. Hence, $\dim \rho =\dim \delta$ and  $i_W(\alpha \rho \perp \beta \delta)\geq \dim B>\frac{\dim \rho}{2}$, which implies that $\rho\cong \delta$ since the Witt index $i_W(\alpha \rho \perp \beta \delta)$ is always a power of $2$ \cite[Th. 3.7]{Laghribi2011}. Obviously, if $B$ is a SPN then it is a Pfister neighbor.
\vskip1.5mm
Recall from \cite{Kato1982} the Kato's isomorphism $e^n:I^nF/I^{n+1}F \longrightarrow \nu_F(n)$ given on generators by:$$e^n(\left< \! \left<a_1, \cdots, a_n\right> \! \right>_b+I^{n+1}F)=\frac{\dif a_1}{a_1}\wedge \cdots \wedge \frac{\dif a_n}{a_n}.$$

The symbol length (or simply the length) of an element $\theta\in \nu_F(n)$ is the smallest number of $n$-logarithmic symbols $\frac{\dif a_1}{a_1}\wedge\cdots \wedge\frac{\dif a_n}{a_n}$ needed to write it.

An Albert bilinear form is a $6$-dimensional bilinear form whose determinant is trivial.

\medskip

\begin{lem}
Let $\gamma$ be an Albert bilinear form and $\tau\in BP_2F$ be such that $\gamma \perp \tau \in I^3F$. Then, $\gamma$ is isotropic.
\label{ll2}
\end{lem}

\begin{proof} (1) If $\tau$ is isotropic, then it is metabolic, and thus $\gamma\in I^3F$. By the Hauptsatz, the form $\gamma$ is metabolic, in particular it is isotropic.

(2) If $\tau$ is anisotropic, then we get by the previous case that $\gamma_{F(\tau)}$ is metabolic. It follows from Proposition \ref{pnd2} that $\gamma$ is isotropic.\qed
\end{proof}

We give a characterization of SPN of dimension $5$. This looks like the characterization of $5$-dimensional quadratic Pfister neighbors (due to Knebusch in characteristic not $2$ \cite[Page 10]{Knebusch1977}, and the second author in characteristic $2$ \cite[Proposition 3.2]{Laghribi2002}).
\medskip
\begin{prop}
Let $B$ be an anisotropic $F$-bilinear form of dimension $5$. The following statements are equivalent:\\(1) $B$ is a SPN.\\(2) $B\cong a\left< \! \left<b, c\right> \! \right>_b \perp \left<d\right>_b$ for suitable $a, b, c, d\in F^*$.\\(3) The invariant $e^2(B\perp \left<\det B\right>_b+I^3F)$ has length $1$.
\label{p3}
\end{prop}

\begin{proof} Let $d\in F^*$ be such that $\det B=d.F^{*2}$.

(1) $\Longrightarrow$ (2) Suppose that $B$ is a SPN of $\pi\in BP_3F$. Then, we have $x\pi \cong B\perp \left<y, z, yzd\right>_b$ for suitable scalars $x, y, z\in F^*$. Hence, $B\perp \left<d\right>_b \perp d\tau \in I^3F$, where $\tau=\left< \! \left<dy, dz\right> \! \right>_b$. We conclude by Lemma \ref{ll2} that $B\perp \left<d\right>_b$ is isotropic, and thus $B\cong B'\perp \left<d\right>_b$ for some bilinear form $B'$ of dimension $4$ and trivial determinant, as desired.

(2) $\Longrightarrow$ (3) Suppose that $B\cong a\left< \! \left<b, c\right> \! \right>_b \perp \left<d\right>_b$. Clearly, we have $e^2(B\perp \left<d\right>_b+I^3F)=\frac{\dif b}{b}\wedge \frac{\dif c}{c}$, which is of length $1$ because the anisotropy of $\left< \! \left<b, c\right> \! \right>_b$ implies that $\frac{\dif b}{b}\wedge \frac{\dif c}{c}\neq 0$. 

(3) $\Longrightarrow$ (1) Suppose that $e^2(B\perp \left<d\right>_b+I^3F)$ has length $1$. Then, there exists an anisotropic $2$-fold bilinear Pfister form $\tau$ such that $e^2(B\perp \left<d\right>_b+I^3F)=e^2(\tau+I^3F)$. Hence, $B\perp \left<d\right>\perp \tau \in I^3F$ using the isomorphism $e^2$. Consequently, $B\perp d\tau' \in I^3F$, where $\tau'$ is the pure part of $\tau$. Then, $B\perp d\tau'$ is similar to a $3$-fold bilinear Pfister form, and thus $B$ is a SPN.\qed
\end{proof}

\section{On $K$-minimal bilinear forms up to dimension 5} 
Throughout this section we take $Q=\left<1, a, b\right>$ an anisotropic totally singular $F$-quadratic form of dimension $3$, and $K=F(Q)$ its function field. 

\begin{lem} (\cite[Lemma 3.7]{Laghribi2012}) Let $B$ be an anisotropic bilinear form over $F$. If $\psi=\left<a_1, \cdots, a_n\right>$ is a subform of $\widetilde{B}$, then there exists a bilinear form $C$ over $F$ such that $C\subset B$ and $\widetilde{C}\cong \psi$. Explicitly, we can take $C=\left<b_1, \cdots, b_n\right>_b$, where $b_i=a_i+\sum_{j=1}^{i-1}a_ix_i^2$ for suitable $x_1, \cdots, x_{i-1}\in F$ (read $b_1=a_1$).
\label{l1}
\end{lem}

The following corollary is immediate.

\begin{cor}
Let $B$ be an anisotropic $F$-bilinear form. Then, $B$ is $K$-minimal iff $\widetilde{B}$ is $K$-minimal.
\label{c1}
\end{cor}

We recall the isotropy results that we need for the classification of $K$-minimal bilinear forms of dimension at most $5$. 

\begin{thm} (\cite[Prop. 1.1 and Th. 1.2]{Laghribi:Rehmann2013})
Let $B$ be an anisotropic $F$-bilinear form such that $\dim B\leq 4$ or $\dim B=5$ and $\nd_F(\widetilde{B})=16$. Then, $B_K$ is isotropic iff $Q$ is similar to a subform of $\widetilde{B}$.
\label{t2}
\end{thm}

As a corollary we get:

\begin{cor}
Let $B$ be an anisotropic $F$-bilinear form of dimension $\leq 5$. If $B_K$ is isotropic, then $B$ is $K$-minimal when $\dim B=3$ or ($\dim B=5$ and $\nd_F(\widetilde{B})=8$).
\label{p4}
\end{cor}

\begin{proof} Suppose that $\dim B\leq 5$ and $B$ is $K$-minimal. The case $\dim B=2$ is excluded as $B_K$ is anisotropic.

-- If $\dim B=3$, then obviously $B$ is $K$-minimal since any subform of $B$ of dimension $2$ is anisotropic over $K$.

-- If $\dim B=4$, then $B_K$ is isotropic iff $Q$ is similar to a subform of $\widetilde{B}$ (Theorem \ref{t2}). Hence, $B$ is not $K$-minimal (Corollary \ref{c1}).

-- If $\dim B=5$. In this case, $\nd_F(\widetilde{B})\in \{8, 16\}$. If $\nd_F(\widetilde{B})= 16$, then $Q$ is similar to a subform of $\widetilde{B}$ (Theorem \ref{t2}). Hence, $B$ is not $K$-minimal when $\nd_F(\widetilde{B})=16$.\qed
\end{proof}

\begin{cor} Let $B$ be an anisotropic $F$-bilinear form of dimension $\leq 5$. If $B$ is isotropic over $K$ but not $K$-minimal, then $Q$ is similar to a subform of $\widetilde{B}$.
\label{cmin}
\end{cor}

\begin{proof}
Since $B$ is not $K$-minimal, there exists $C$ a subform of $B$ such that $\dim C<\dim B$ and $C_K$ is isotropic. It follows from Theorem \ref{t2} that $Q$ is similar to a subform of $\widetilde{C}$.\qed
\end{proof}

\begin{lem}
Let $\pi_1\in BP_mF$ and $\pi_2\in BP_nF$ with $m\leq n$. Suppose that $\pi_1'$ is similar to a subform of $\pi_2$, where $\pi_1'$ denotes the pure part of $\pi_1$. Then, $\pi_2\cong \pi_1\otimes \tau$ for some $\tau \in BP_{n-m}F$.
\label{ll}
\end{lem}

\begin{proof}
We have $i_W(\pi_1\perp \alpha\pi_2)\geq 2^m-1$ for some $\alpha \in F^*$. It follows from \cite[Th. 3.7]{Laghribi2011} that this Witt index is equal to $\dim \pi_1$, and the forms $\pi_1$ and $\pi_2$ are $m$-linked, which means that $\pi_2\cong \pi_1\otimes \tau$ for some $\tau \in BP_{n-m}F$.\qed
\end{proof}

\section{Proof of Theorem \ref{tp3}} Let $B$ be an anisotropic bilinear form of dimension $5$, and $Q=\left<1, a, b\right>$ an anisotropic totally singular quadratic form of dimension $3$. Let $K=F(Q)$ be the function field of $Q$.\vskip1.5mm

(2) $\Longrightarrow$ (1) Suppose there exists an $F$-bilinear form $C$ of dimension $5$ which is a SPN of a bilinear Pfister form $\rho:=\left< \! \left< a, b, c\right> \! \right>_b$ and satisfies the two conditions:
\begin{itemize}
\item[(a)] $\widetilde{B}\cong \widetilde{C}$.
\item[(b)] For any $u, v\in F^2(a, b)$ such that $\left<u, v, uv\right>_b$ is similar to a subform of $\rho$, the invariant $e^2(C\perp \left<\det C\right>_b\perp\left< \! \left<u, v\right> \! \right>_b+I^3F)$ has length $2$.
\end{itemize}

Let $d\in F^*$ be such that $\det C=d.F^{*2}$. The form $C_K$ is isotropic because $\rho_K$ is isotropic and $C$ is a SPN of $\rho$. In particular, $B_K$ is isotropic. 

Suppose that $B$ is not $K$-minimal. Then, $C$ is not $K$-minimal as well because $\widetilde{B}\cong \widetilde{C}$ (Corollary \ref{c1}). It follows from Corollary \ref{cmin} that $Q=\left<1, a, b\right>$ is similar to a subform of $\widetilde{C}$. By Lemma \ref{l1}, we conclude that $p\left<1, a+q^2, b+ar^2+s^2\right>_b$ is a subform of $C$ for suitable $p\neq 0, r, s\in F$. In particular, $\left<a+q^2, b+ar^2+s^2, (a+q^2)(b+ar^2+s^2)\right>_b$ is similar to a subform of $\rho$, and thus our hypothesis (b) above implies that $e^2(C\perp \left<d\right>_b\perp\left< \! \left<a+q^2, b+ar^2+s^2\right> \! \right>_b+I^3F)$ has length $2$. 

Let $u=a+q^2$, $v=b+ar^2+s^2$. It is easy to see that $C\perp \left<d\right>_b \sim  p\left< \! \left<u, v\right> \! \right>_b \perp \tau$ for some $\tau\in GBP_2F$. Consequently, the invariant $e^2(C\perp\left<d\right>\perp \left< \! \left<u, v\right> \! \right>_b+I^3F)$ has length at most $1$, a contradiction.

(1) $\Longrightarrow$ (2) Suppose that $B$ is $K$-minimal. Then, we get by Corollary \ref{p4} that $\nd_F(\widetilde{B})=8$. It follows from Proposition \ref{p4} that $\widetilde{B}$ is quasi-Pfister neighbor of a quasi-Pfister form $\pi$. Since $\widetilde{B}_K$ is isotropic, it follows that $\pi_{K}$ is quasi-hyperbolic. Hence, $\pi \cong \left< \! \left<a, b, c\right> \! \right>$ for some $c\in F^*$ \cite[Th. 1.5]{Laghribi2004}. There exists a bilinear form $C$ of dimension $5$ similar to a subform of $\rho:=\left< \! \left<a, b, c\right> \! \right>_b$ such that $\widetilde{B}\cong \widetilde{C}$ (Lemma \ref{l1}). In particular, the form $C$ is a SPN of $\rho$. Modulo a scalar, we may write $C\cong \left< \! \left<x, y\right> \! \right>_b \perp \left<z\right>_b$ for suitable $x, y, z\in F^*$ (Proposition \ref{p3}). 

Let $u, v\in F^2(a, b)$ be such that $\left<u, v, uv\right>_b$ is similar to a subform of $\rho$. The form $\left<u, v, uv\right>$ is anisotropic because $\rho$ is anisotropic. On the one hand, the condition $u, v\in F^2(a, b)$ implies that $\left<1, u, v\right>$ becomes isotropic over $K$, which gives by Theorem \ref{t2} that $Q$ is similar to $\left<1, u, v\right>$ and thus $K=F(\left<1, u, v\right>)$. On the other hand, using Lemma \ref{ll}, we get $\rho\cong \left< \! \left<u, v, w\right> \! \right>_b$ for some $w\in F^*$. Hence, without loss of generality, we may suppose $\left< \! \left<a, b\right> \! \right>_b\cong \left< \! \left<u, v\right> \! \right>_b$ for the rest of the proof.

We have $e^2(C\perp \left<z\right>_d \perp \left< \! \left<a, b\right> \! \right>_b+I^3F)=\frac{\dif x}{x}\wedge \frac{\dif y}{y} + \frac{\dif a}{a}\wedge\frac{\dif b}{b}$. Suppose that this invariant has length $\leq 1$. Then, there exists a $2$-fold bilinear Pfister form $\tau$ such that $e^2(C \perp \left<z\right>_d\perp \left< \! \left<a, b\right> \! \right>_b+I^3F)=e^2(\tau+I^3F)$, which implies that $\left< \! \left<x, y\right> \! \right>_b \perp \left< \! \left<a, b\right> \! \right>_b \perp \tau \in I^3_qF$. It follows from Lemma \ref{ll2} that the Albert form $\left<a, b, ab, x, y, xy\right>_b$ is isotropic. Hence, the forms $\left< \! \left<x, y\right> \! \right>_b$ and $\left< \! \left<a, b\right> \! \right>_b$ are $1$-linked, meaning that $\left< \! \left<x, y\right> \! \right>_b\cong \left< \! \left<e, r\right> \! \right>_b$ and $\left< \! \left<a, b\right> \! \right>_b\cong \left< \! \left<f, r\right> \! \right>_b$ for suitable $e, f, r\in F^*$. By the uniqueness of the pure part of bilinear Pfister forms, we get $\left<a, b, ab\right>_b \cong \left<f, r, fr\right>_b$, and thus $K=F(\left<1, f, r\right>)$. We continue with some arguments similar to those used by Faivre in his proof. 

Hence, without loss of generality, we may suppose that $C\cong \left< \! \left<e, b\right> \! \right>_b \perp \left<z\right>_b$. So the form $C$ is a SPN of $\left< \! \left<e, b,z\right> \! \right>_b$. But $C$ is also a SPN of $\left< \! \left<a, b, c\right> \! \right>_b$, it follows that $\left< \! \left<e, b,z\right> \! \right>_b\cong \left< \! \left<a, b, c\right> \! \right>_b$. Adding in both sides the form $\left<1, b\right>_b$, we get$$\M(1)\perp \M(b)\perp \left<z, e, ez\right>_b\otimes \left<1, b\right>_b \cong \M(1)\perp \M(b)\perp \left<c, a, ac\right>_b\otimes \left<1, b\right>_b,$$where $\M(\alpha)$ denotes the metabolic plane given by the matrix $\begin{pmatrix}\alpha & 1\\1 & 0\end{pmatrix}$ for any $\alpha \in F$. By the uniqueness of the anisotropic part, we get$$\left<z, e, ez\right>_b\otimes \left<1, b\right>_b \cong \left<c, a, ac\right>_b\otimes \left<1, b\right>_b.$$Adding in both sides $a\left<1, b\right>_b$, we get$$a\left< \! \left<ea, b\right> \! \right>_b \perp z\left< \! \left<e, b\right> \! \right>_b \cong \M(a)\perp \M(ab)\perp c\left< \! \left<a, b\right> \! \right>_b.$$In particular, $a\left< \! \left<ea, b\right> \! \right>_b \perp z\left< \! \left<e, b\right> \! \right>_b$ is isotropic, and thus there exist $r\in D_F(\left< \! \left<ea, b\right> \! \right>_b)$ and $s\in D_F(\left< \! \left<e, b\right> \! \right>_b)$ such that $ar=zs$. We have
\begin{eqnarray*}
C & \cong & \left< \! \left<e, b\right> \! \right>_b \perp \left<z\right>_b\\& \cong & s\left< \! \left<e, b\right> \! \right>_b \perp \left<z\right>_b\\ & \cong & s\left<1, b\right>_b \perp D,
\end{eqnarray*}
where $D=es\left<1, b\right>_b \perp \left<z\right>_b$. Let $\beta:=as\left<1, b\right>_b\perp D$. Then, we have
\begin{eqnarray*}
\beta & \cong & as\left< \! \left<ea, b\right> \! \right>_b \perp \left<z\right>_b\\ & \cong & ars\left< \! \left<ea, b\right> \! \right>_b \perp \left<z\right>_b\\ & \cong & z\left< \! \left<ea, b\right> \! \right>_b \perp \left<z\right>_b.
\end{eqnarray*}
Hence, $\beta\cong \M(z)\perp \widetilde{\beta}\perp \left<zb\right>_b$, where $\widetilde{\beta}=zea\left<1, b\right>_b$. Now, we have
\begin{eqnarray*}
bC\perp \widetilde{\beta}& \sim & s\left< \! \left<e, b\right> \! \right>_b \perp \beta\\& \sim & s\left< \! \left<e, b\right> \! \right>_b \perp as\left<1, b\right>_b \perp es\left<1, b\right>_b \perp \left<z\right>_b\\& \sim & s\left< \! \left<a, b\right> \! \right>_b \perp \left<z\right>_b.
\end{eqnarray*}
This shows that $bC\perp \widetilde{\beta}$ is isotropic. Then, there exist bilinear forms $C_1$ and $C_2$ of dimension $4$ and $1$, respectively, such that $C_1\subset bC$, $C_2\subset \widetilde{\beta}$ and $C_1\perp C_2\cong s\left< \! \left<a, b\right> \! \right>_b \perp \left<z\right>_b$. Then, $i_W((C_1\perp C_2)_K)=2$ and thus $(C_1)_K$ is isotropic, meaning that $C$ is not $K$-minimal. Since $\widetilde{C}\cong \widetilde{B}$, it follows that $B$ is not $K$-minimal, a contradiction.\qed

Using Theorem \ref{tp2}, we provide an example of a $K$-minimal bilinear form of dimension $5$, where $K$ is the function field of a singular conic. The form we choose in our example  is inspired from \cite[Proposition 4.1]{CQM}.

\begin{example}
Let $F_0$ be a field of characteristic $2$, and $k=F_0(a, b, c)$ the rational function field in the indeterminates $a, b, c$ over $F_0$. Let $B=c\left<1, a+b\right>_b \perp b\left<1, a\right>_b \perp \left<1\right>_b$ and $Q=\left<1, a, c\right>$. Then, $B$ is $k(Q)$-minimal. 
\end{example}

\begin{proof}
Using the isometry $\left< \! \left<a, b\right> \! \right>_b\cong \left< \! \left<ab, a+b\right> \! \right>_b$, we get 
\begin{eqnarray*}
B\perp \left<a\right>_b \perp abc\left<1, a+b\right>_b & = & c\left< \! \left< ab, a+b\right> \! \right>_b \perp \left< \! \left<a, b\right> \! \right>_b\\ &\cong & c\left< \! \left< ab, a+b\right> \! \right>_b \perp \left< \! \left<ab, a+b\right> \! \right>_b\\& = & \left< \! \left<c, ab, a+b\right> \! \right>_b\\& \cong &\left< \! \left<a, b, c\right> \! \right>_b.
\end{eqnarray*}
Hence, $B$ is a SPN of $\left< \! \left<a, b, c\right> \! \right>_b$, and thus $B_K$ is isotropic. Moreover, $B$ is anisotropic over $k$ since $\left< \! \left<a, b, c\right> \! \right>_b$ is also anisotropic.

Let $u, v\in k^2(a, c)$ be such that $\left<u, v, uv\right>_b$ is a subform of $\left< \! \left<a, b, c\right> \! \right>_b$. We have 
$$e^2(B\perp \left<\det B\right>_b \perp \left< \! \left<u, v\right> \! \right>_b+I^3F)=\frac{\dif(a+b)}{a+b}\wedge\frac{\dif (bc)}{bc}+\frac{\dif u}{u}\wedge\frac{\dif v}{v}.$$If this invariant is of length $\leq 1$, then the Albert form $\left<u, v, uv, a+b, bc, (a+b)bc\right>_b$ is isotropic over $k$. Since the forms $\left<u, v, uv\right>_b$ and $\left<a+b, bc, (a+b)bc\right>_b$ are anisotropic over $k$, there exists $\alpha \in k^*$ represented by both forms. Let us write$$\alpha=uL^2+vM^2+uvN^2=(a+b)S^2+bcT^2+(a+b)bcU^2$$for suitable $L, M, N, S, T, U\in k$. Hence, we have$$b(S^2+cT^2+acU^2)=uL^2+vM^2+uvN^2+aS^2+c(bU)^2.$$The right hand side in this equality and the factor $S^2+cT^2+acU^2$ belong to $k^2(a, c)$, but since $b\not\in k^2(a, c)$, we necessarily have $S^2+cT^2+acU^2=0$. Since $\left<1, c, ac\right>$ is anisotropic over $k$, it follows that $S=T=U=0$ and thus $\alpha=0$, a contradiction. 

Consequently, $e^2(B\perp \left<\det B\right>_b \perp \left< \! \left<u, v\right> \! \right>_b+I^3F)$ is of length $2$. Since all the conditions of Theorem \ref{tp3} are satisfied (taking for $C$ the form $B$ itself), we conclude that $B$ is $k(Q)$-minimal.
\end{proof}
\medskip

\subsection*{\sc Acknowledgement.} This work was done in the framework of the project ``IEA of CNRS'' between the Artois University and the Academic College of Tel-Aviv-Yaffo. The two authors are grateful for the support of both institutions and CNRS.

\end{sloppypar}

\begin{thebibliography}{99}
\bibitem{B} R. Baeza, \newblock{\it Quadratic forms over semilocal rings,} LNM 655, Springer-Verlag Berlin Heidelberg New York 1978.

\bibitem{CQM} A. Chapman, A. Qu\'eguiner-Mathieu, \newblock{\it Minimal quadratic forms for the function field of a conic in characteristic 2,} Contemporary Math. {\bf 800} (2024), 117-128.

\bibitem{Dolphin:Laghribi2017} A. Dolphin, A. Laghribi, {\it Isotropy of 5-dimensional quadratic forms over the function field of a quadric in characteristic 2,} Communications in Algebra 45 (7) (2017), 3034-3044.

\bibitem{Faivre2006} F. Faivre \newblock{\it Liaison des formes de Pfister et corps de fonctions de quadriques en caract\'eristique 2,} PhD Thesis, Universit\'e de Franche-Comt\'e, 2006.

\bibitem{HLVG} D. W. Hoffmann, D. Lewis, J. Van Geel, {\it Minimal forms for function fields of conics,} Proceedings of Symposia in Pure Mathematics 58.2 (1995), p. 227-237.

\bibitem{Hoffmann:Laghribi2004} Hoffmann D.W. and Laghribi A., \textit{Quadratic forms and Pfister neighbors in characteristic 2}, Trans. Amer. Math. Soc., \textbf{356} (2004), no. 10, 4019-4053.

\bibitem{Hoffmann:Laghribi2006} Hoffmann D.W., Laghribi A., \textit{Isotropy of quadratic forms over the function field of a quadric in characteristic 2.}, J. Algebra, \textbf{295} (2006), no. 2, 362-386.

\bibitem{Kato1982} K. Kato, \textit{Symmetric bilinear forms, quadratic forms and Milnor {$K$}-theory in characteristic two}, Invent. Math., \textbf{66} (1982), 493-510.

\bibitem{Knebusch1973} M. Knebusch, \newblock{\it Specialization of quadratic and bilinear forms, and a norm theorem,} Acta Arithmetica {\bf XXIV} (1973), 279-299.

\bibitem{Knebusch1977} M. Knebusch, \newblock{\it Generic splitting of quadratic forms. II,} Proc. London Math. Soc. (3) 34 (1977), no. 1, 1-31.

\bibitem{Laghribi2002} A. Laghribi, \textit{Certaines formes quadratiques de dimension au plus 6 et corps des fonctions en caract\'eristique 2}, Israel J. Math., \textbf{129} (2002), 317-361.

\bibitem{Laghribi2004} A. Laghribi, \newblock{\it Quasi-hyperbolicity of totally singular quadratic forms,} Contemporary Mathematics 344 (2004), 237-248.

\bibitem{Laghribi2005} A. Laghribi, \newblock{\it Witt kernels of function field extensions in characteristic 2.}, J. Pure Appl. Algebra, \textbf{199} (2005), no. 1-3, 167-182.

\bibitem{Laghribi2007} A. Laghribi, \newblock{\it Sur le d\'eploiement des formes bilin\'eaires en caract\'eristique 2,} Pacific J. Math. {\bf 232} (2007), 207-232.

\bibitem{Laghribi2011} A. Laghribi, \newblock{\it Les formes bilin\'eaires et quadratiques bonnes de hauteur 2 en caract\'eristique 2,} Math. Z. 269 (2011), 671-685.

\bibitem{Laghribi2012} A. Laghribi, \newblock{\it Isotropie d'une forme bilin\'eaire d'Albert sur le corps de fonctions d'une quadrique en caract\'eristique 2,}  J. Algebra. 355 (2012), 1-8.

\bibitem{Laghribi:Mukhija2024} A. Laghribi, D. Mukhija, \newblock{\it The excellence of function fiels of conics,} Proc. Amer. Math. Soc. 152(5) (2024), 1915-1923.

\bibitem{Laghribi:Rehmann2013} A. Laghribi, U. Rehmann, \newblock{\it Bilinear forms of dimension less than or equal to 5 and function fields of quadrics in characteristic
2,} Math. Nachr. 286 (2013), 1180--1190.

\bibitem{MTW} P. Mammone, J.-P. Tignol, A. Wadsworth, \newblock{\it Fields of characteristic 2 with prescribed u-invariants,} Math. Annalen  {\bf 290} (1990), 109-128.

\end{thebibliography}
\end{document}